\documentclass{article}
\title{\textbf{A generalization of the Askey-Wilson relations\\using a projective geometry}}
\usepackage{authblk}
\author[]{Ian Seong}
\date{}
\usepackage {parskip}
\usepackage [margin=2.5cm] {geometry}
\usepackage{enumitem}
\usepackage[utf8]{inputenc}
\usepackage[T1]{fontenc}
\usepackage{geometry}
\usepackage{amsmath, amssymb, amsthm}
\usepackage{hyperref}
\usepackage{graphicx}
\usepackage{mathtools}
\usepackage{relsize}
\usepackage{bm}
\usepackage{esvect}
\usepackage{circuitikz}
\usepackage{hhline}
\usepackage{enumitem}
\usepackage{comment}
\usepackage{ytableau}
\usepackage{pdflscape}
\usepackage{mathabx}

\newtheoremstyle{dotless}{}{}{\itshape}{}{\bfseries}{}{ }{}
\theoremstyle{dotless}
\newtheorem{theorem}{Theorem}[section]
\newtheorem{corollary}[theorem]{Corollary}
\newtheorem{lemma}[theorem]{Lemma}

\newtheoremstyle{dotlessdef}{}{}{}{}{\bfseries}{}{ }{}
\theoremstyle{dotlessdef}
\newtheorem{definition}[theorem]{Definition}

\newtheorem{note}[theorem]{Note}

\begin{document}
\maketitle

\begin{abstract}
In this paper, we present a generalization of the Askey-Wilson relations that involves a projective geometry. A projective geometry is defined as follows. Let $h>k\geq 1$ denote integers. Let $\mathbb{F}_{q}$ denote a finite field with $q$ elements. Let $\mathcal{V}$ denote an $(h+k)$-dimensional vector space over $\mathbb{F}_{q}$. Let the set $P$ consist of the subspaces of $\mathcal{V}$. The set $P$, together with the inclusion partial order, is a poset called a projective geometry. We define a matrix $A\in \text{Mat}_{P}(\mathbb{C})$ as follows. For $u,v\in P$, the $(u,v)$-entry of $A$ is $1$ if each of $u,v$ covers $u\cap v$, and $0$ otherwise. Fix $y\in P$ with $\dim y=k$. We define a diagonal matrix $A^*\in \text{Mat}_{P}(\mathbb{C})$ as follows. For $u\in P$, the $(u,u)$-entry of $A^{*}$ is $q^{\dim(u\cap y)}$. We show that 
\begin{align*}
    &A^2A^{*}-\bigl(q+q^{-1}\bigr)AA^{*}A+A^{*}A^{2}-\mathcal{Y}\bigl(AA^{*}+A^{*}A\bigr)-\mathcal{P} A^{*}=\Omega A+G,\\
            &A^{*2}A-\bigl(q+q^{-1}\bigr) A^*AA^*+AA^{*2}=\mathcal{Y}A^{*2}+\Omega A^{*}+G^{*},
\end{align*} where $\mathcal{Y}, \mathcal{P}, \Omega, G, G^*$ are matrices in $\text{Mat}_{P}(\mathbb{C})$ that commute with each of $A, A^*$. We give precise formulas for $\mathcal{Y}, \mathcal{P}, \Omega, G, G^*$.
\end{abstract}

\section{Introduction}
 There is a family of graphs said to be distance-regular \cite{BBIT, BCN}. Some distance-regular graphs satisfy a condition called the $Q$-polynomial property \cite[p.~135]{BCN}. Let $\Gamma$ denote a $Q$-polynomial distance-regular graph. There are some well-known parameters $\beta,\gamma,\gamma^*,\varrho, \varrho^*$ that are used to describe the $Q$-polynomial structure \cite[Lemma~5.4]{ter3}. Let $A$ denote the adjacency matrix of $\Gamma$. Pick a vertex $y$ of $\Gamma$, and let $A^{*}=A^{*}(y)$ denote the dual adjacency matrix of $\Gamma$ with respect to $y$ \cite[p.~121]{Liang}. The algebra $T$ generated by $A,A^{*}$ is often called the Terwilliger algebra of $\Gamma$ with respect to $y$ \cite[p.~364]{ter1}. Let $W$ denote a thin irreducible $T$-module \cite[p.~366]{ter1}. By \cite[Lemma~14.9]{Worawannotai}, there exist complex scalars $\omega=\omega(W),\eta=\eta(W),\eta^*=\eta^*(W)$ such that the following relations hold on $W$: 
\begin{align}
        A^2A^{*}-\beta AA^{*}A+A^{*}A^{2}-\gamma(AA^{*}+A^{*}A)-\varrho A^*&=\gamma^*A^2+\omega A+\eta I,\label{ask1}\\
        A^{*2}A-\beta A^{*}AA^{*}+AA^{*2}-\gamma^{*}(A^{*}A+AA^{*})-\varrho^{*} A&=\gamma A^{*2}+\omega A^{*}+\eta^{*} I.\label{ask2}
    \end{align}
The relations (\ref{ask1}), (\ref{ask2}) are called the Askey-Wilson relations \cite[Theorem~1.5]{Vidunas}. 

We now assume that every irreducible $T$-module is thin. Then according to Worawannotai \cite[Lemma~14.12]{Worawannotai}, the algebra $T$ contains some central elements $\Omega,G, G^*$ such that
\begin{align*}
        A^2A^{*}-\beta AA^{*}A+A^{*}A^{2}-\gamma(AA^{*}+A^{*}A)-\varrho A^*&=\gamma^*A^2+\Omega A+G,\\
        A^{*2}A-\beta A^{*}AA^{*}+AA^{*2}-\gamma^{*}(A^{*}A+AA^{*})-\varrho^{*} A&=\gamma A^{*2}+\Omega A^{*}+G^{*}.
    \end{align*}
    In \cite[Section~20,~21]{Worawannotai}, Worawannotai finds $\Omega, G, G^*$ explicitly under the assumption that $\Gamma$ is a dual polar graph. 

    In the present paper, we obtain a generalization of the Askey-Wilson relations that is roughly analogous to the Worawannotai result, with $\Gamma$ replaced by a projective geometry. To explain our result, we first recall the definition of a projective geometry. Let $h>k\geq 1$ denote integers. Let $\mathbb{F}_{q}$ denote a finite field with $q$ elements. Let $\mathcal{V}$ denote an $(h+k)$-dimensional vector space over $\mathbb{F}_{q}$. Let the set $P$ consist of the subspaces of $\mathcal{V}$. The set $P$, together with the inclusion partial order, is a poset called a projective geometry.

    Let $\text{Mat}_{P}(\mathbb{C})$ denote the $\mathbb{C}$-algebra consisting of the matrices with rows and columns indexed by $P$ and all entries in $\mathbb{C}$. Define $A\in \text{Mat}_{P}(\mathbb{C})$ as follows. For $u,v\in P$, the $(u,v)$-entry of $A$ is
    \begin{equation*}
        A_{u,v}=\begin{cases}
            1&\text{if each of $u,v$ covers $u\cap v$,}\\
            0&\text{otherwise.}
        \end{cases}
    \end{equation*} Fix $y\in P$ with $\dim y=k$. Define a diagonal matrix $A^*\in \text{Mat}_{P}(\mathbb{C})$ as follows. For $u\in P$, the $(u,u)$-entry of $A^{*}$ is
    \begin{equation*}
        \bigl(A^{*}\bigr)_{u,u}=q^{\dim(u\cap y)}.
    \end{equation*}
    We show that
\begin{align}
            &A^2A^{*}-\bigl(q+q^{-1}\bigr)AA^{*}A+A^{*}A^{2}-\mathcal{Y}\bigl(AA^{*}+A^{*}A\bigr)-\mathcal{P} A^{*}=\Omega A+G,\label{gen1}\\
            &A^{*2}A-\bigl(q+q^{-1}\bigr) A^*AA^*+AA^{*2}=\mathcal{Y}A^{*2}+\Omega A^{*}+G^{*},\label{gen2}
        \end{align}
        where $\mathcal{Y}, \mathcal{P}, \Omega, G, G^*$ are matrices in $\text{Mat}_{P}(\mathbb{C})$ that commute with each of $A, A^*$. We give a precise formula for $\mathcal{Y}, \mathcal{P}, \Omega, G, G^*$ in the main body of the paper. The relations (\ref{gen1}), (\ref{gen2}) are the main results of the paper.
        
        To prove our results, we make heavy use of the work of Watanabe \cite{Watanabe}. We now recall some definitions and results from that paper. Define diagonal matrices $K_1, K_2 \in \text{Mat}_{P}(\mathbb{C})$ as follows. For $u \in P$ the $(u,u)$-entry is
        \begin{equation*}
            \bigl(K_1\bigl)_{u,u}=q^{\frac{k}{2}-i}, \qquad \qquad \bigl(K_2\bigl)_{u,u}=q^{j-\frac{h}{2}},
        \end{equation*}
        where $i=\dim (u\cap y)$ and $i+j=\dim u$. The matrices $K_1, K_2$ are invertible. We define matrices $L_1,L_2\in \text{Mat}_{P}(\mathbb{C})$ as follows. For $u,v\in P$ their $(u,v)$-entries are
    \begin{align*}
        \bigl(L_1\bigr)_{u,v}&=\begin{cases}
            1&\text{if $u\subseteq v$ and $\dim u=\dim v-1$ and $\dim (u\cap y)=\dim (v\cap y)-1$,}\\
            0&\text{otherwise,}
        \end{cases}\\
        \bigl(L_2\bigr)_{u,v}&=\begin{cases}
            1&\text{if $u\subseteq v$ and $\dim u=\dim v-1$ and $\dim (u\cap y)=\dim (v\cap y)$,}\\
            0&\text{otherwise.}
        \end{cases}
    \end{align*}
    Define $R_1=L_1^{t}$ and $R_2=L_2^{t}$, where $t$ is the matrix-transpose. Let $\mathcal{H}$ denote the subalgebra of $\text{Mat}_{P}(\mathbb{C})$ generated by $L_1, L_2, R_1, R_2, K_1^{\pm1}, K_2^{\pm1}$. In \cite[Section~7]{Watanabe}, Watanabe gives many relations between the $\mathcal{H}$-generators. In \cite[Section~8]{Watanabe}, Watanabe classifies up to isomorphism the irreducible $\mathcal{H}$-modules. For each irreducible $\mathcal{H}$-module $W$, he gives a basis for $W$ and the action of the $\mathcal{H}$-generators on that basis. In \cite[Section~9]{Watanabe}, Watanabe gives a generating set $\Lambda_0, \Lambda_1, \Lambda_2$ for the center of $\mathcal{H}$.       
        
    We now summarize how we use the Watanabe results to obtain (\ref{gen1}), (\ref{gen2}). The basic idea is to evaluate $A,A^*$ in (\ref{gen1}), (\ref{gen2}) using the following reduction. Note that \begin{equation*}
        A^*=q^{\frac{k}{2}}K_1^{-1}.
    \end{equation*}
    We show that 
    \begin{equation*}
        A=R+L+F^{0}+F^{+}+F^{-},
    \end{equation*}
    where
    \begin{align*}
        R&=L_1R_2,\\
        L&=L_2R_1,\\
        F^{0}&=L_1R_1-R_1L_1+(q-1)^{-1}\Bigl(q^{\frac{h+k}{2}}K_1^{-1}K_2-q^{\frac{k}{2}}K_1-q^{\frac{h}{2}}K_2+I\Bigr),\\
        F^{+}&=L_2R_2-q^{\frac{k}{2}}(q-1)^{-1}K_1\Bigl(q^{\frac{h}{2}}K_2^{-1}-I\Bigr),\\
        F^{-}&=R_1L_1-q^{\frac{h}{2}}(q-1)^{-1}K_2\Bigl(q^{\frac{k}{2}}K_1^{-1}-I\Bigr).
    \end{align*}

    We remark that $R,L,F^{0}, F^{+}, F^{-}$ correspond to some orbits $\mathcal{B}_{xy}, \mathcal{C}_{xy},\mathcal{A}_{xy}^{0},\mathcal{A}_{xy}^{+},\mathcal{A}_{xy}^{-}$ in \cite[Definition~6.1,~6.6]{Seong2}. We adjust the central elements $\Lambda_0, \Lambda_1, \Lambda_2$ to obtain some elements $\Omega_0, \Omega_1, \Omega_2$ that generate the center of $\mathcal{H}$. We express $F^{0}, F^{+}, F^{-}$ as polynomials in $K_1^{\pm1}, K_2^{\pm1}, \Omega_0,\Omega_1,\Omega_2$. In this way, we express everything in (\ref{gen1}), (\ref{gen2}) in terms of  $R, L, K_1^{\pm1}, K_2^{\pm1},\Omega_0,\Omega_1,\Omega_2$. This yields (\ref{gen1}), (\ref{gen2}) after some algebraic manipulation.

    We have been discussing (\ref{gen1}), (\ref{gen2}). We also obtain the following related results that may be of independent interest. We display the action of $\Omega_0,\Omega_1,\Omega_2, \mathcal{Y}, \mathcal{P}, \Omega, G, G^*$ on the irreducible $\mathcal{H}$-modules. By \cite[Theorem~8.9]{Watanabe}, each irreducible $\mathcal{H}$-module is determined up to isomorphism by certain parameters $\alpha,\beta,\rho$. In some previous literature \cite{Liang, ter3}, the action of $R,L,A,A^*$ on the irreducible $\mathcal{H}$-modules is expressed in terms of some parameters called the endpoint $\nu$, dual endpoint $\mu$, diameter $d$, and an auxiliary parameter $e$. Following \cite[Section~3]{Liang}, we explain how to convert from $\alpha,\beta,\rho$ to $\nu,\mu,d,e$.
    
    The paper is organized as follows. In Section 2 we give preliminaries on the projective geometry and the algebra $\mathcal{H}$. In Section 3 we talk about the matrices $F^{0},F^{+}, F^{-}$. In Section 4 we discuss the center of $\mathcal{H}$. In Sections 5 and 6, we discuss the matrices $R, L, A, A^*$. In Section 7, we find $\mathcal{Y}, \mathcal{P}, \Omega, G, G^*$ and show that the matrices $A,A^*$ satisfy (\ref{gen1}) and (\ref{gen2}). In Section 8, we explain how to convert from $\alpha,\beta,\rho$ to $\nu,\mu,d,e$. Section 9 is an appendix that contains some relations involving the $\mathcal{H}$-generators.

\section{Projective geometry $P$ and algebra $\mathcal{H}$}
\label{proj}

    Let $h,k$ denote integers such that $h>k\geq 1$. Let $\mathbb{F}_{q}$ denote a finite field with $q$ elements. Let $\mathcal{V}$ denote an $(h+k)$-dimensional vector space over $\mathbb{F}_{q}$. Let the set $P$ consist of the subspaces of $\mathcal{V}$. The set $P$, together with the inclusion partial order, is a poset called a \emph{projective geometry}. For $u,v\in P$, we say that $v$ \emph{covers} $u$ whenever $v\supseteq u$ and $\dim v=\dim u+1$. 

    Let $\text{Mat}_{P}(\mathbb{C})$ denote the $\mathbb{C}$-algebra consisting of the matrices with rows and columns indexed by $P$ and all entries in $\mathbb{C}$. Let $V$ denote the $\mathbb{C}$-vector space consisting of the column vectors with rows indexed by $P$ and all entries in $\mathbb{C}$. The algebra $\text{Mat}_{P}(\mathbb{C})$ acts on $V$ by left multiplication. The $\text{Mat}_{P}(\mathbb{C})$-module $V$ is called \emph{standard}. For $u\in P$, define the column vector $\widehat{u}\in V$ that has $u$-entry $1$ and all other entries $0$. The vectors $\bigl\{\widehat{u}\mid u\in P\bigr\}$ form a basis for $V$.

    For $0\leq \ell\leq h+k$ let the set $P_{\ell}$ consist of the $\ell$-dimensional subspaces of $\mathcal{V}$. We have a partition
    \begin{equation*}
        P=\bigcup_{\ell=0}^{h+k}P_{\ell}.
    \end{equation*}
    
    \begin{definition}
    \label{e*l}
        For $0\leq \ell\leq h+k$ define a diagonal matrix $E^{*}_{\ell}\in \text{Mat}_{P}(\mathbb{C})$ as follows. For $u\in P$, the $(u,u)$-entry of $E^{*}_{\ell}$ is 
    \begin{equation*}
        \bigl(E^{*}_{\ell}\bigr)_{u,u}=\begin{cases}
            1&\text{if $u\in P_{\ell}$},\\
            0&\text{if $u\not\in P_{\ell}$.}
        \end{cases}
    \end{equation*}
    \end{definition}
    
    We have
    \begin{equation*}
        I=\sum_{\ell=0}^{h+k}E^{*}_{\ell},
    \end{equation*}
    \begin{equation*}
        E^{*}_{i}E^{*}_{j}=\delta_{i,j}E^{*}_{i}\qquad \qquad 0\leq i,j\leq h+k.
    \end{equation*}
    
    For $0\leq \ell\leq h+k$,
    \begin{equation*}
        E^{*}_{\ell}V=\text{Span}\bigl\{\widehat{u}\mid u\in P_{\ell}\bigr\}.
    \end{equation*}
    The following sum is direct:
    \begin{equation*}
        V=\sum_{\ell=0}^{h+k}E^{*}_{\ell}V.
    \end{equation*}

    For the rest of this paper, we fix $y\in P_{k}$. For $0 \leq i\leq k$ and $0 \leq j\leq h$ define
    \begin{equation*}
        P_{i,j}=\bigl\{u\in P\mid \dim (u\cap y)=i, \dim u=i+j\bigr\}.
    \end{equation*}

    We have a partition
    \begin{equation*}
        P=\bigcup_{i=0}^{k}\bigcup_{j=0}^{h}P_{i,j}.
    \end{equation*}

    For notational convenience, for integers $r,s$ we define $P_{r,s}$ to be empty unless $0\leq r\leq k$ and $0\leq s\leq h$.

    In the diagram below, we illustrate the set $P$ and the subsets $P_{i,j}$ $(0\leq i\leq k,\; 0\leq j\leq h)$.
    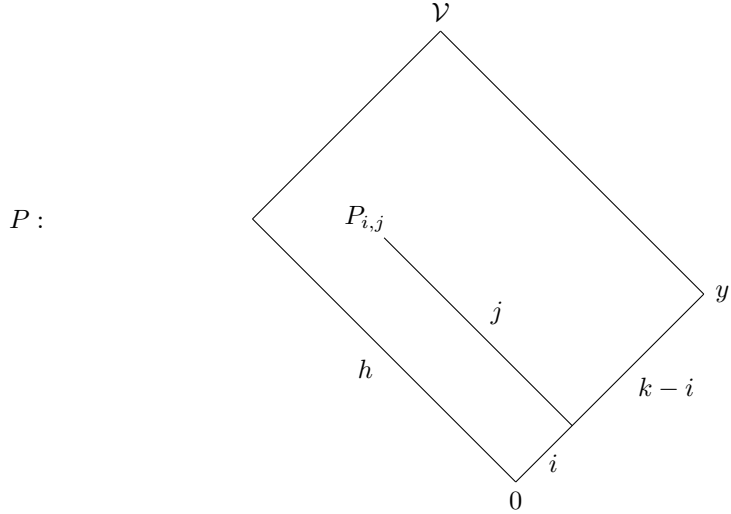
\begin{figure}[!ht]
\centering
{%
\begin{circuitikz}
\tikzstyle{every node}=[font=\normalsize]
\draw [short] (2.5,-3.5) -- (-1,0);
\draw [short] (-1,0) -- (1.5,2.5);
\draw [short] (1.5,2.5) -- (5,-1);
\draw [short] (5,-1) -- (2.5,-3.5);
\draw [short] (0.75,-0.25) -- (3.25,-2.75);
\node [font=\normalsize] at (1.5,2.75) {$\mathcal{V}$};
\node [font=\normalsize] at (5.25,-1) {$y$};
\node [font=\normalsize] at (2.5,-3.75) {$0$};
\node [font=\normalsize] at (3,-3.25) {$i$};
\node [font=\normalsize] at (2.25,-1.25) {$j$};
\node [font=\normalsize] at (0.5,0) {$P_{i,j}$};
\node [font=\normalsize] at (4.5,-2.25) {$k-i$};
\node [font=\normalsize] at (0.5,-2) {$h$};
\node [font=\normalsize] at (-4,0) {$P:$};
\node [font=\normalsize] at (8,0) {$\;$};
\end{circuitikz}
}%
\caption{The projective geometry $P$ and the location of $P_{i,j}$.}
\label{projective}
\end{figure}

    \begin{definition}
        For $0 \leq i\leq k$ and $0 \leq j\leq h$, define a diagonal matrix $E^{*}_{i,j}\in \text{Mat}_P(\mathbb{C})$ as follows. For $u\in P$, the $(u,u)$-entry of $E^{*}_{i,j}$ is
    \begin{equation*}
        \bigl(E^{*}_{i,j}\bigr)_{u,u}=\begin{cases}
            1&\text{if }u\in P_{i,j},\\
            0&\text{if $u\notin P_{i,j}$}.
        \end{cases}
    \end{equation*}  
    For notational convenience, for integers $r,s$ we define $E^{*}_{r,s}=0$ unless $0 \leq r\leq k$ and $0 \leq s\leq h$.
    \end{definition}

    For $0\leq i\leq k$ and $0\leq j\leq h$, we have
    \begin{equation*}
        E^{*}_{i,j}V={\rm{Span}} \lbrace{ \widehat u} \mid u \in P_{i,j}\rbrace.
    \end{equation*}

     The following sum is direct:
    \begin{equation*}
        V=\sum_{i=0}^{k}\sum_{j=0}^{h} E^{*}_{i,j}V.
    \end{equation*}
    
    We now describe how the sets    \begin{equation*}
        P_{\ell}\qquad \qquad \qquad 0\leq \ell\leq h+k
    \end{equation*} are related to the sets
    \begin{equation*}
        P_{i,j}\qquad \qquad 0\leq i\leq k\qquad \qquad 0\leq j\leq h.
    \end{equation*}
    For $0\leq \ell\leq h+k$ we have a partition
    \begin{equation*}
        P_{\ell}=\bigcup_{i,j}P_{i,j},
    \end{equation*}
    where the union is over the ordered pairs $(i,j)$ such that $0\leq i\leq k$ and $0\leq j\leq h$ and $i+j=\ell$. 
    
    For $0\leq \ell\leq h+k$,
    \begin{align*}
        E^{*}_{\ell}&=\sum_{i,j}E^{*}_{i,j},\\
        E^{*}_{\ell}V&=\sum_{i,j}E^{*}_{i,j}V,\qquad \qquad (\text{direct sum})
    \end{align*}
    where the sums are over the ordered pairs $(i,j)$ such that $0\leq i\leq k$ and $0\leq j\leq h$ and $i+j=\ell$. 
    
    Earlier, we described the covering relation on $P$. Next, we give a refinement of the covering relation.
    
    \begin{lemma}
    \label{coverlem}
        Let $u,v\in P$ such that $v$ covers $u$. Write
        \begin{equation*}
            u\in P_{i,j}, \qquad \qquad v\in P_{r,s}.
        \end{equation*}
        Then either (i) $r=i+1$ and $s=j$, or (ii) $r=i$ and $s=j+1$.
    \end{lemma}

    \begin{proof}
        By linear algebra, 
        \begin{equation*}
            u\cap v\in P_{a,b},
        \end{equation*}
        where $a\leq \min\{i,r\}$ and $b\leq \min\{j,s\}$.
        Since $v$ covers $u$, we have $u\subseteq v$ and $r+s=i+j+1$. 
        Since $u\cap v=u\in P_{i,j}$, we have $i\leq r$ and $j\leq s$. The result follows.
    \end{proof}
        
    \begin{definition}
    \label{slashcover}
        Referring to Lemma \ref{coverlem},  we say that \emph{$v$ $\slash$-covers $u$} whenever (i) holds, and \emph{$v$ $\backslash$-covers $u$} whenever (ii) holds.
    \end{definition}

    We illustrate Definition \ref{slashcover} using the diagrams below.
    \begin{figure}[!ht]
\centering
{%
\begin{circuitikz}
\tikzstyle{every node}=[font=\normalsize]
\draw [short] (2.5,-3.5) -- (-1,0);
\draw [short] (-1,0) -- (1.5,2.5);
\draw [short] (1.5,2.5) -- (5,-1);
\draw [short] (5,-1) -- (2.5,-3.5);
\draw [short] (0.9,-0.1) -- (1.1,0.1);
\node [font=\normalsize] at (1.5,2.75) {$\mathcal{V}$};
\node [font=\normalsize] at (5.25,-1) {$y$};
\node [font=\normalsize] at (2.5,-3.75) {$0$};
\node [font=\normalsize] at (1.25,0.25) {$v$};
\node [font=\normalsize] at (0.75,-0.25) {$u$};
\node [font=\normalsize] at (4,-2.5) {$k$};
\node [font=\normalsize] at (0.5,-2) {$h$};
\node [font=\normalsize] at (2,-4.75) {Figure 2A: $v$ $\slash$-covers $u$.};
\end{circuitikz}
\hspace{0.25cm}
\begin{circuitikz}
\draw [short] (9.5,-3.5) -- (6,0);
\draw [short] (6,0) -- (8.5,2.5);
\draw [short] (8.5,2.5) -- (12,-1);
\draw [short] (12,-1) -- (9.5,-3.5);
\draw [short] (8.4,0.1) -- (8.6,-0.1);

\node [font=\normalsize] at (8.5,2.75) {$\mathcal{V}$};
\node [font=\normalsize] at (12.25,-1) {$y$};
\node [font=\normalsize] at (9.5,-3.75) {$0$};
\node [font=\normalsize] at (8.25,0.25) {$v$};
\node [font=\normalsize] at (8.75,-0.25) {$u$};
\node [font=\normalsize] at (11,-2.5) {$k$};
\node [font=\normalsize] at (7.5,-2) {$h$};
\node [font=\normalsize] at (9,-4.75) {Figure 2B: $v$ $\backslash$-covers $u$.};
\end{circuitikz}
}%
\end{figure}

\renewcommand{\thefigure}{\arabic{figure}}

\setcounter{figure}{2}

In what follows, we use the notation
    \begin{equation*}
        [m]=\frac{q^m-1}{q-1}\qquad \qquad (m\in \mathbb{Z}).
    \end{equation*}

\begin{lemma}{\cite[Lemma~5.1]{Watanabe}}
    \label{cover}
        For $0\leq i\leq k$ and $0\leq j\leq h$, the following (i)--(iv) hold:
        \begin{enumerate}[label=(\roman*)]
            \item each element in $P_{i,j}$ $\slash$-covers exactly  $q^j[i]$ elements;

            \item each element in $P_{i,j}$ $\backslash$-covers exactly $[j]$ elements;

            \item each element in $P_{i,j}$ is $\slash$-covered by exactly $[k-i]$ elements;

            \item each element in $P_{i,j}$ is $\backslash$-covered by exactly $q^{k-i}[h-j]$ elements.
        \end{enumerate}
    \end{lemma}

    Next we define a subalgebra $\mathcal{K}$ of $\text{Mat}_{P}(\mathbb{C})$.

    \begin{definition}{\cite[Definition~6.1]{Watanabe}}
        The matrices 
        \begin{equation*}
            E^{*}_{i,j}\qquad \qquad 0\leq i\leq k\qquad \qquad 0\leq j\leq h
        \end{equation*} 
        form a basis for a commutative subalgebra of $\text{Mat}_{P}(\mathbb{C})$. Denote this subalgebra by $\mathcal{K}$.
    \end{definition}

    Our next goal is to describe a generating set for $\mathcal{K}$. For the rest of the paper, $q^{1/2}$ denotes the positive square root of $q$.
    
    \begin{definition}
    \label{kdef}
        We define diagonal matrices $K_1, K_2 \in \text{Mat}_{P}(\mathbb{C})$ as follows. For $u \in P$ the $(u,u)$-entry is
        \begin{equation*}
            \bigl(K_1\bigl)_{u,u}=q^{\frac{k}{2}-i}, \qquad \qquad \bigl(K_2\bigl)_{u,u}=q^{j-\frac{h}{2}},
        \end{equation*}
        where $u\in P_{i,j}$.  
    \end{definition}

    Note that $K_1,K_2$ are invertible. By \cite[Prop.~6.3]{Watanabe}, the algebra $\mathcal{K}$ is generated by $K_1^{\pm 1},K_2^{\pm 1}$.

    \begin{lemma}
    \label{kv}
        For $0\leq i\leq k$ and $0\leq j\leq h$, the subspace $E^{*}_{i,j}V$ is a common eigenspace for $K_1^{\pm1}, K_2^{\pm1}$. The corresponding eigenvalues are given in the table below.
        \begin{center}
        \begin{tabular}{c|c}
            Element in $\mathcal{K}$ & Eigenvalue corresponding to $E^{*}_{i,j}V$ \\
            \hline \\
            $K_1$ & $q^{\frac{k}{2}-i}$\\ \\
            $K_1^{-1}$ & $q^{i-\frac{k}{2}}$\\ \\
            $K_2$ & $q^{j-\frac{h}{2}}$\\ \\
            $K_2^{-1}$ & $q^{\frac{h}{2}-j}$\\ \\
        \end{tabular}
        \end{center}
    \end{lemma}
    \begin{proof}
        Immediate from Definition \ref{kdef}.
    \end{proof}
        
    Next we define some matrices in $\text{Mat}_{P}(\mathbb{C})$ that will be useful. 
    \begin{definition}
    \label{l1l2r1r2def}
        We define matrices $L_1,L_2,R_1,R_2\in \text{Mat}_{P}(\mathbb{C})$ as follows. For $u,v\in P$ their $(u,v)$-entries are
    \begin{align*}
        \bigl(L_1\bigr)_{u,v}&=\begin{cases}
            1&\text{if }v\text{ $\slash$-covers }u,\\
            0&\text{if }v\text{ does not $\slash$-cover }u,
        \end{cases}\\
        \bigl(L_2\bigr)_{u,v}&=\begin{cases}
            1&\text{if }v\text{ $\backslash$-covers }u,\\
            0&\text{if }v\text{ does not $\backslash$-cover }u,
        \end{cases}\\
        \bigl(R_1\bigr)_{u,v}&=\begin{cases}
            1&\text{if }u\text{ $\slash$-covers }v,\\
            0&\text{if }u\text{ does not $\slash$-cover }v,
        \end{cases}\\
        \bigl(R_2\bigr)_{u,v}&=\begin{cases}
            1&\text{if }u\text{ $\backslash$-covers }v,\\
            0&\text{if }u\text{ does not $\backslash$-cover }v.
        \end{cases}
    \end{align*}
    
    Note that $R_1=L_1^{t}$ and $R_2=L_2^{t}$, where $t$ is the matrix-transpose. 
    \end{definition}

    \begin{lemma}
    \label{l1l2r1r2v}
        For $v\in P$,
        \begin{align*}
            L_1\widehat{v}&=\sum_{\text{$v$ $\slash$-covers $u$}}\widehat{u},&L_2\widehat{v}=\sum_{\text{$v$ $\backslash$-covers $u$}}\widehat{u},\\
            R_1\widehat{v}&=\sum_{\text{$u$ $\slash$-covers $v$}}\widehat{u}, &R_2\widehat{v}=\sum_{\text{$u$ $\backslash$-covers $v$}}\widehat{u}.\\
        \end{align*}
    \end{lemma}

    \begin{proof}
        Immediate from Definition \ref{l1l2r1r2def}.
    \end{proof}

    \begin{lemma}
    \label{lrv}
        For $0\leq i\leq k$ and $0\leq j\leq h$,
        \begin{align}
            L_1E^{*}_{i,j}V&\subseteq E^{*}_{i-1,j}V, &L_2E^{*}_{i,j}V\subseteq E^{*}_{i,j-1}V, \label{lev}\\
            R_1E^{*}_{i,j}V&\subseteq E^{*}_{i+1,j}V, &R_2E^{*}_{i,j}V\subseteq E^{*}_{i,j+1}V. \label{rev}
        \end{align}
    \end{lemma}
    \begin{proof}
        Immediate from Lemma \ref{l1l2r1r2v}.
    \end{proof}    
    
    \begin{definition}
    \label{hdef}
        Let $\mathcal{H}$ denote the subalgebra of $\text{Mat}_{P}(\mathbb{C})$ generated by $L_1, L_2, R_1, R_2, K_1^{\pm 1}, K_2^{\pm 1}$. 
    \end{definition}

    By construction, the vector space $V$ is an $\mathcal{H}$-module. Let $W$ denote an $\mathcal{H}$-submodule in $V$. We say that $W$ is \emph{irreducible} whenever $W$ does not contain an $\mathcal{H}$-submodule besides $0$ or $W$. Note that $\mathcal{H}$ is closed under the conjugate-transpose map. Therefore the $\mathcal{H}$-module $V$ is a direct sum of irreducible $\mathcal{H}$-modules.

    \begin{lemma}{\cite[Lemma~8.2,~8.5,~8.11]{Watanabe}}
    \label{typelem}
        Let $W$ denote an irreducible $\mathcal{H}$-module. Then there exist integers $\alpha,\beta,\rho$ such that the following (i), (ii) hold:
        \begin{enumerate}[label=(\roman*)]
            \item 
            \begin{equation}
            \label{abcondition}
                0\leq \rho,\qquad \qquad 0\leq \alpha\leq \frac{k-\rho}{2}, \qquad \qquad 0\leq \beta\leq \frac{h-\rho}{2};
            \end{equation}

            \item for $0\leq i\leq k$ and $0\leq j\leq h$,
            \begin{equation*}
            \dim E^{*}_{i,j}W=\begin{cases}
                1&\text{if }\alpha\leq i\leq k-\rho-\alpha,\;\rho+\beta\leq j\leq h-\beta,\\
                0&\text{otherwise}.
            \end{cases}
        \end{equation*}
        \end{enumerate}
    \end{lemma}

    \begin{definition}
        Let $W$ denote an irreducible $\mathcal{H}$-module. Referring to Lemma \ref{typelem}, we call the triple $(\alpha,\beta,\rho)$ the \emph{type of }$W$. 
    \end{definition}

    \begin{lemma}{\cite[p.~133]{Liang}}
    \label{unique}
        For each triple $(\alpha,\beta,\rho)$ of integers that satisfy (\ref{abcondition}), there exists an irreducible $\mathcal{H}$-module of type $(\alpha,\beta,\rho)$. This $\mathcal{H}$-module is unique up to isomorphism of $\mathcal{H}$-modules.
    \end{lemma}

    We illustrate Lemma \ref{typelem} using the diagram below. 

    \begin{figure}[!ht]
\centering
{%
\begin{circuitikz}
\tikzstyle{every node}=[font=\normalsize]
\draw [short] (2.5,-3.5) -- (-1,0);
\draw [short] (-1,0) -- (1.5,2.5);
\draw [short] (1.5,2.5) -- (5,-1);
\draw [short] (5,-1) -- (2.5,-3.5);
\draw [ color=red , fill=pink, rotate around={45:(2,2)}] (0.5,1.25) rectangle (-0.25,-0.5);
\draw [<->, >=Stealth] (2.1,-0.2) -- (3.15,0.85);
\draw [<->, >=Stealth] (0.2,1.2) -- (1.22,0.18);
\draw [<->, >=Stealth] (3.7,-2.3) -- (2.47,-1.07);
\draw [<->, >=Stealth] (1.55,-0.75) -- (0.65,-1.65);
\node [font=\normalsize] at (1.5,2.85) {$\mathcal{V}$};
\node [font=\normalsize] at (5.25,-1) {$y$};
\node [font=\normalsize] at (2.5,-3.75) {$0$};
\node [font=\normalsize] at (1.75,0.75) {$W$};
\node [font=\normalsize] at (-4,0) {$P:$};
\node [font=\normalsize] at (8,0) {$\;$};
\node [font=\scriptsize, rotate around={45:(0,0)}] at (1.3,-1.35) {$\alpha$};
\node [font=\scriptsize, rotate around={45:(0,0)}] at (2.85,0.15) {$\rho+\alpha$};
\node [font=\scriptsize, rotate around={-45:(0,0)}] at (2.95,-1.85) {$\rho+\beta$};
\node [font=\scriptsize, rotate around={-45:(0,0)}] at (0.55,0.5) {$\beta$};

\end{circuitikz}
}%
\caption{The irreducible $\mathcal{H}$-module $W$ is represented by the red rectangle.}
\label{locationabp}
\end{figure}
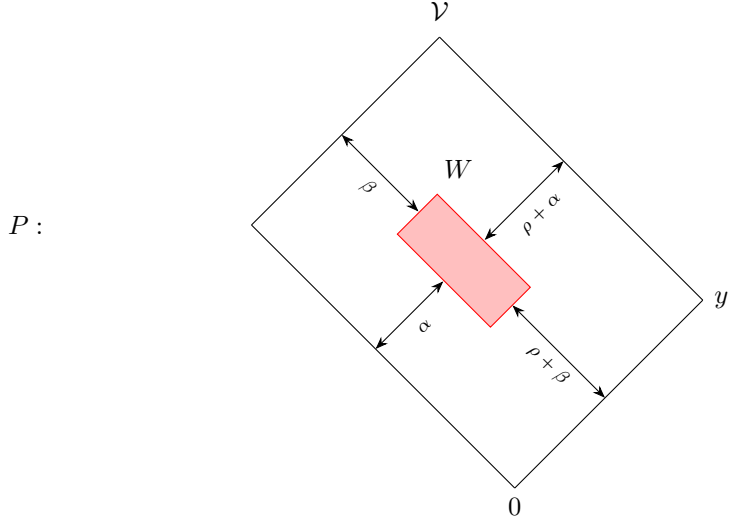
 \newpage
      
    Our next goal is to describe the action of the $\mathcal{H}$-generators on the irreducible $\mathcal{H}$-modules. For the rest of this section, let $W$ denote an irreducible $\mathcal{H}$-module of type $(\alpha,\beta,\rho)$.

    \begin{lemma}
        \label{klem}
        For $\alpha\leq i\leq k-\rho-\alpha$ and $\rho+\beta\leq j\leq h-\beta$, the subspace $E^{*}_{i,j}W$ is a common eigenspace for $K_1^{\pm 1}, K_2^{\pm 1}$. 
        The corresponding eigenvalues are given in the table below.
        \begin{center}
        \begin{tabular}{c|c}
            Element in $\mathcal{H}$ & Eigenvalue corresponding to $E^{*}_{i,j}W$ \\
            \hline \\
            $K_1$ & $q^{\frac{k}{2}-i}$\\ \\
            $K_1^{-1}$ & $q^{i-\frac{k}{2}}$\\ \\
            $K_2$ & $q^{j-\frac{h}{2}}$\\ \\
            $K_2^{-1}$ & $q^{\frac{h}{2}-j}$\\ \\
        \end{tabular}
        \end{center}
    \end{lemma}

    \begin{proof}
        Immediate from Lemma \ref{kv}.
    \end{proof}

    \begin{lemma}
    \label{lrw}
        For $\alpha\leq i\leq k-\rho-\alpha$ and $\rho+\beta\leq j\leq h-\beta$,
        \begin{equation}
        \label{le}
            L_1E^{*}_{i,j}W\subseteq E^{*}_{i-1,j}W,\qquad \qquad L_2E^{*}_{i,j}W\subseteq E^{*}_{i,j-1}W, 
        \end{equation}
        \begin{equation}
        \label{re}
            R_1E^{*}_{i,j}W\subseteq E^{*}_{i+1,j}W, \qquad \qquad R_2E^{*}_{i,j}W\subseteq E^{*}_{i,j+1}W.
        \end{equation}
    \end{lemma}
    \begin{proof}
        Immediate from Lemma \ref{lrv}.
    \end{proof}    

    The following result is a variation on \cite[Prop.~8.7]{Watanabe}.

    \begin{lemma}
    \label{l1l2r1r2w}
        The irreducible $\mathcal{H}$-module $W$ has a basis
        \begin{equation}
        \label{wnm}
            w_{i,j} \qquad \qquad \alpha\leq i\leq k-\rho-\alpha \qquad \qquad \rho+\beta\leq j\leq h-\beta
        \end{equation}
        with the following properties. For $\alpha\leq i\leq k-\rho-\alpha$ and $\rho+\beta\leq j\leq h-\beta$,
        \begin{align}
        w_{i,j}&\in E^{*}_{i,j}W,\nonumber\\
            L_1 w_{i,j}&=q^{\frac{\rho+\alpha+\beta+i+j-1}{2}}[k-\rho-\alpha-i+1]w_{i-1,j},\label{l1w}\\
            L_2w_{i,j}&=q^{k-\frac{\rho+\alpha-\beta+i-j+1}{2}}[h-\beta-j+1]w_{i,j-1},\label{l2w}\\
            R_1w_{i,j}&=q^{-\frac{\rho-\alpha+\beta+i-j}{2}}[i-\alpha+1]w_{i+1,j},\label{r1w}\\
            R_2w_{i,j}&=q^{\frac{\rho+\alpha+\beta-i-j}{2}}[j-\rho-\beta+1]w_{i,j+1}.\label{r2w}
        \end{align}
        In the above lines, we interpret $w_{r,s}=0$ unless $\alpha\leq r\leq k-\rho-\alpha$ and $\rho+\beta\leq s\leq h-\beta$.
    \end{lemma}

    \begin{proof}
        Let
        \begin{equation*}
            w'_{n,m}\qquad \qquad 0\leq n\leq k-2\alpha-\rho \qquad \qquad  0\leq m\leq h-2\beta-\rho
        \end{equation*}
        denote the basis from \cite[Prop.~8.7]{Watanabe}. Note that
        \begin{equation*}
            w'_{n,m}\in E^{*}_{\alpha+n,\;\rho+\beta+m}W.
        \end{equation*} 
        Define 
        \begin{equation}
        \label{w'nm}
            w_{i,j}=q^{-\frac{(i-\alpha)(j-\rho-\beta)}{2}}w'_{i-\alpha,\;j-\rho-\beta}\qquad \qquad \alpha\leq i\leq k-\rho-\alpha \qquad \qquad \rho+\beta\leq j\leq h-\beta.
        \end{equation}
        Then we have
        \begin{equation*}
            w_{i,j}\in E^{*}_{i,j}W.
        \end{equation*}
        By Lemma \ref{typelem}(ii), the vectors in (\ref{w'nm}) form a basis for $W$. Equations (\ref{l1w})--(\ref{r2w}) follow from (\ref{w'nm}) and \cite[Prop.~8.7]{Watanabe}.  
    \end{proof} 

    \begin{definition}
        A basis for $W$ that satisfies Lemma \ref{l1l2r1r2w} will be called \emph{standard}.
    \end{definition}
    
    \begin{lemma}
    \label{double}
        For $\alpha\leq i\leq k-\rho-\alpha$ and $\rho+\beta\leq j\leq h-\beta$, the subspace $E^{*}_{i,j}W$ is invariant under each of
        \begin{equation*}
            L_1R_1, \qquad \qquad R_1L_1,\qquad \qquad L_2R_2, \qquad \qquad R_2L_2.
        \end{equation*}
        The corresponding eigenvalues are given in the table below.
        \begin{center}
        \begin{tabular}{c|c}
            Element in $\mathcal{H}$ & Eigenvalue corresponding to $E^{*}_{i,j}W$ \\
            \hline \\
            $L_1R_1$ & $q^{\alpha+j}[i-\alpha+1][k-\rho-\alpha-i]$\\ \\
            $R_1L_1$ & $q^{\alpha+j}[i-\alpha][k-\rho-\alpha-i+1]$\\ \\
            $L_2R_2$ & $q^{k+\beta-i}[j-\rho-\beta+1][h-\beta-j]$\\ \\
            $R_2L_2$ & $q^{k+\beta-i}[j-\rho-\beta][h-\beta-j+1]$            
        \end{tabular}
        \end{center}
    \end{lemma}

    \begin{proof}
        The first assertion follows from (\ref{le}), (\ref{re}). For the second assertion, combine (\ref{l1w})--(\ref{r2w}).
    \end{proof}

    \section{Matrices $F^{0}, F^{+}, F^{-}$ and their action on irreducible $\mathcal{H}$-modules}

    In this section we define some matrices $F^{0},F^{+},F^{-}$ in $\mathcal{H}$, and describe their action on the irreducible $\mathcal{H}$-modules.

    \begin{definition}
    \label{ftildedef}
        Define matrices $F^{0},F^{+},F^{-}\in \text{Mat}_{P}(\mathbb{C})$ as follows. For $u,v\in P$ their $(u,v)$-entries are
        \begin{align*}
            \bigl(F^{0}\bigr)_{u,v}&=\begin{cases}
                1&\text{if $u+v$ $\slash$-covers $u$ and $u+v$ $\slash$-covers $v$ and $u$ $\backslash$-covers $u\cap v$ and $v$ $\backslash$-covers $u\cap v$},\\
                0&\text{otherwise},
            \end{cases}\\
            \bigl(F^{+}\bigr)_{u,v}&=\begin{cases}
                1&\text{if $u+v$ $\backslash$-covers $u$ and $u+v$ $\backslash$-covers $v$},\\
                0&\text{otherwise},
            \end{cases}\\
            \bigl(F^{-}\bigr)_{u,v}&=\begin{cases}
                1&\text{if $u$ $\slash$-covers $u\cap v$ and $v$ $\slash$-covers $u\cap v$},\\
                0&\text{otherwise}.
            \end{cases}
        \end{align*}
    \end{definition}

    Our next goal is to show that $F^{0},F^{+},F^{-}$ are contained in $\mathcal{H}$. To do this, we write each of $F^{0},F^{+},F^{-}$ in terms of the generators of $\mathcal{H}$.
    
    \begin{lemma}
    \label{f0}
        The matrix $F^{0}$ satisfies
        \begin{align}
            F^{0}&=L_1R_1-R_1L_1+(q-1)^{-1}\Bigl(q^{\frac{h+k}{2}}K_1^{-1}K_2-q^{\frac{k}{2}}K_1-q^{\frac{h}{2}}K_2+I\Bigr)\label{f01}\\
            &=R_2L_2-L_2R_2+(q-1)^{-1}\Bigl(q^{\frac{h+k}{2}}K_1K_2^{-1}-q^{\frac{k}{2}}K_1-q^{\frac{h}{2}}K_2+I\Bigr)\label{f02}.
        \end{align}
    \end{lemma}

    \begin{proof}
        We first prove (\ref{f01}). For $u,v\in P$ we calculate the $(u,v)$-entry of the right-hand side of (\ref{f01}). By Lemma \ref{cover}(iii), the $(u,v)$-entry of $L_1R_1$ is
        \begin{equation}
        \label{L1R1entry}
            \bigl(L_1R_1\bigr)_{u,v}=\begin{cases}
                [k-i]&\text{if }u=v \text{ and }u\in P_{i,j}\;(0\leq i\leq k,\; 0\leq j\leq h),\\
                1&\text{if $u+v$ $\slash$-covers $u$ and $u+v$ $\slash$-covers $v$},\\
                0&\text{otherwise}.
            \end{cases}
        \end{equation}
        By Lemma \ref{cover}(i), the $(u,v)$-entry of $R_1L_1$ is
        \begin{equation}
            \label{R1L1entry}
            \bigl(R_1L_1\bigr)_{u,v}=\begin{cases}
                q^j[i]&\text{if }u=v \text{ and }u\in P_{i,j}\;(0\leq i\leq k,\; 0\leq j\leq h),\\
                1&\text{if $u$ $\slash$-covers $u\cap v$ and $v$ $\slash$-covers $u\cap v$},\\
                0&\text{otherwise}.
            \end{cases}
        \end{equation}
        By Definition \ref{kdef} the matrix
        \begin{equation}
        \label{weird1}
            (q-1)^{-1}\Bigl(q^{\frac{h+k}{2}}K_2K_1^{-1}-q^{\frac{k}{2}}K_1-q^{\frac{h}{2}}K_2+I\Bigr)
        \end{equation}
        is diagonal. The $(u,u)$-entry of (\ref{weird1}) is equal to $q^j[i]-[k-i]$, where $u\in P_{i,j}$. By these comments, the $(u,v)$-entry of the right-hand side of (\ref{f01}) is equal to the $(u,v)$-entry of the left-hand side. The result follows.

        We have proved (\ref{f01}). Equation (\ref{f02}) follows from (\ref{f01}) and Lemma \ref{appendix4} in the appendix.
    \end{proof}

    \begin{lemma}
    \label{f+}
        The matrix $F^{+}$ satisfies
        \begin{equation}
        \label{f+def}
            F^{+}=L_2R_2-q^{\frac{k}{2}}(q-1)^{-1}K_1\Bigl(q^{\frac{h}{2}}K_2^{-1}-I\Bigr).
        \end{equation}
    \end{lemma}

    \begin{proof}
        For $u,v\in P$ we calculate the $(u,v)$-entry of the right-hand side of (\ref{f+def}). By Lemma \ref{cover}(iv), the $(u,v)$-entry of $L_2R_2$ is
        \begin{equation}
            \label{L2R2entry}
            \bigl(L_2R_2\bigr)_{u,v}=\begin{cases}
                q^{k-i}[h-j]&\text{if }u=v \text{ and }u\in P_{i,j}\;(0\leq i\leq k,\; 0\leq j\leq h),\\
                1&\text{if $u+v$ $\backslash$-covers $u$ and $u+v$ $\backslash$-covers $v$},\\
                0&\text{otherwise}.
            \end{cases}
        \end{equation}
        By Definition \ref{kdef} the matrix
        \begin{equation}
        \label{weird2}
            q^{\frac{k}{2}}(q-1)^{-1}K_1\Bigl(q^{\frac{h}{2}}K_2^{-1}-I\Bigr)
        \end{equation}
        is diagonal. The $(u,u)$-entry of (\ref{weird2}) is equal to $q^{k-i}[h-j]$, where $u\in P_{i,j}$. By these comments, the $(u,v)$-entry on the right-hand side of (\ref{f+def}) is equal to the $(u,v)$-entry on the left-hand side. The result follows.
    \end{proof}

    \begin{lemma}
    \label{f-}
        The matrix $F^{-}$ satisfies
        \begin{equation}
        \label{f-def}
            F^{-}=R_1L_1-q^{\frac{h}{2}}(q-1)^{-1}\Bigl(q^{\frac{k}{2}}K_1^{-1}-I\Bigr)K_2.
        \end{equation}
    \end{lemma}

    \begin{proof}
        For $u,v\in P$ we calculate the $(u,v)$-entry of the right-hand side of (\ref{f-def}). By Definition \ref{kdef} the matrix
        \begin{equation}
        \label{weird3}
            q^{\frac{h}{2}}(q-1)^{-1}\Bigl(q^{\frac{k}{2}}K_1^{-1}-I\Bigr)K_2
        \end{equation}
        is diagonal. The $(u,u)$-entry of (\ref{weird3}) is $q^j[i]$, where $u\in P_{i,j}$. By these comments and (\ref{R1L1entry}), the $(u,v)$-entry on the right-hand side of (\ref{f-def}) is equal to the $(u,v)$-entry on the left-hand side. The result follows.
    \end{proof}

    \begin{lemma}
    \label{finh}
        The matrices $F^{0}, F^{+}, F^{-}$ are contained in $\mathcal{H}$.
    \end{lemma}

    \begin{proof}
        Immediate from Lemmas \ref{f0}--\ref{f-}.
    \end{proof}
    
    Next we write $L_1R_1, R_1L_1, L_2R_2, R_2L_2$ in terms of $F^{0}, F^{+}, F^{-}, K_1^{\pm 1}, K_2^{\pm 1}$.

    \begin{lemma}
        The following (\ref{l1r1back})--(\ref{r2l2back}) hold:
        \begin{align}
            L_1R_1&=F^{0}+F^{-}+(q-1)^{-1}\Bigl(q^{\frac{k}{2}}K_1-I\Bigr); \label{l1r1back}\\
            R_1L_1&=F^{-}+q^{\frac{h}{2}}(q-1)^{-1}\Bigl(q^{\frac{k}{2}}K_1^{-1}-I\Bigr)K_2;\label{r1l1back}\\
            L_2R_2&=F^{+}+q^{\frac{k}{2}}(q-1)^{-1}K_1\Bigl(q^{\frac{h}{2}}K_2^{-1}-I\Bigr); \label{l2r2back}\\
            R_2L_2&=F^{0}+F^{+}+(q-1)^{-1}\Bigl(q^{\frac{h}{2}}K_2-I\Bigr).\label{r2l2back}
        \end{align}
    \end{lemma}

    \begin{proof}
        Routine from Lemmas \ref{f0}--\ref{f-}.    
    \end{proof}

    Next we define a matrix $F$.

    \begin{definition}
    \label{deff}
        Define
        \begin{equation}
        \label{fdef}
            F=F^{0}+F^{+}+F^{-}.
        \end{equation}
    \end{definition}

    \begin{lemma}
        \label{ftildeeq}
        We have
        \begin{align}
            F&=L_1R_1+L_2R_2-(q-1)^{-1}\Bigl(q^{\frac{h+k}{2}}K_1K_2^{-1}-I\Bigr),\label{fdef2}\\
            &=R_1L_1+R_2L_2-(q-1)^{-1}\Bigl(q^{\frac{h+k}{2}}K_1^{-1}K_2-I\Bigr).\label{fdef3}
        \end{align}
    \end{lemma}

    \begin{proof}
        Equation (\ref{fdef2}) follows from (\ref{l1r1back}), (\ref{l2r2back}), (\ref{fdef}). Equation (\ref{fdef3}) follows from (\ref{fdef2}) and Lemma \ref{appendix4} in the appendix.
    \end{proof}
    
    The following result gives a combinatorial interpretation of $F$.
    
    \begin{lemma}
        The matrix $F$ has the following entries. For $u,v\in P$, the $(u,v)$-entry is
        \begin{equation*}
            F_{u,v} = \begin{cases}
                1&\text{if each of $u,v$ covers $u\cap v$ and $\dim (u\cap y)=\dim (v\cap y)$,}\\
                0&\text{otherwise.}
            \end{cases}
        \end{equation*}
    \end{lemma}

    \begin{proof}
        This is a routine consequence of Lemma \ref{ftildeeq}.
    \end{proof}
    
    Next we describe the action of $F^{0},F^{+},F^{-}, F$ on the irreducible $\mathcal{H}$-modules. Recall the standard module $V$. 

     \begin{lemma}
     \label{fe*}
        For $0\leq i\leq k$ and $0\leq j\leq h$,
        \begin{align}
            &F^{0}E^{*}_{i,j}V\subseteq E^{*}_{i,j}V, &F^{+}E^{*}_{i,j}V\subseteq E^{*}_{i,j}V, \label{fe*1}\\
            &F^{-}E^{*}_{i,j}V\subseteq E^{*}_{i,j}V, &FE^{*}_{i,j}V\subseteq E_{i,j}^{*}V.\label{fe*2}
        \end{align}
    \end{lemma}
    \begin{proof}
        For $F^{0}$ the result is immediate from (\ref{f01}), Definition \ref{kdef}, Lemma \ref{lrv}. For $F^{+}$ and $F^{-}$ the proofs are similar, and omitted. For $F$, the result follows from (\ref{fdef}), (\ref{fe*1}), and the left equation of (\ref{fe*2}).
    \end{proof}

    For the rest of this section, let $W$ denote an irreducible $\mathcal{H}$-module of type $(\alpha,\beta,\rho)$.
    
    \begin{lemma}
    \label{factionlem}
        For $\alpha\leq i\leq k-\rho-\alpha$ and $\rho+\beta\leq j\leq h-\beta$, the subspace $E^{*}_{i,j}W$ is invariant under each of $F^{0},F^{+},F^{-}, F$. 
    \end{lemma}

    \begin{proof}
        Immediate from Lemma \ref{fe*}.
    \end{proof}

    \begin{definition}
        For $\alpha\leq i\leq k-\rho-\alpha$ and $\rho+\beta\leq j\leq h-\beta$ let
        \begin{equation*}
            a_{i,j}^{0}(W),\qquad \qquad a_{i,j}^{+}(W),\qquad \qquad a_{i,j}^{-}(W), \qquad \qquad a_{i,j}(W)
        \end{equation*}
        denote the eigenvalues of $F^{0}, F^{+}, F^{-}, F$ respectively that correspond to $E^{*}_{i,j}W$.
    \end{definition}

    By construction we have 
    \begin{equation*}
    \label{aij}
        a_{i,j}(W)=a_{i,j}^{0}(W)+a_{i,j}^{+}(W)+a_{i,j}^{-}(W).
    \end{equation*}

    \begin{theorem}
    \label{faction}
        For $\alpha\leq i\leq k-\rho-\alpha$ and $\rho+\beta\leq j\leq h-\beta$,
        \begin{align}
            a_{i,j}^{0}(W)&=q^{k-i}[j-\rho]-[j],\label{aij0}\\
            a_{i,j}^{+}(W)&=q^{k-i}\Bigl(q^{\beta+1}[j-\rho-\beta][h-\beta-j]-[\beta]\Bigr),\label{aij+}\\
            a_{i,j}^{-}(W)&=q^{j}\Bigl(q^{\alpha+1}[i-\alpha][k-\rho-\alpha-i]-[\alpha]\Bigr).\label{aij-}
        \end{align}
    \end{theorem}
    
    \begin{proof}
        We prove (\ref{aij0}). Pick $w\in E^{*}_{i,j}W$. By (\ref{f01}),
        \begin{equation}
        \label{f0weq}
            F^{0}w=\biggl(L_1R_1-R_1L_1+(q-1)^{-1}\Bigl(q^{\frac{h+k}{2}}K_2K_1^{-1}-q^{\frac{k}{2}}K_1-q^{\frac{h}{2}}K_2+I\Bigr)\biggr)w.
        \end{equation}
        Evaluate the right-hand side of (\ref{f0weq}) using the $L_1R_1,\;R_1L_1$-entries of the table in Lemma \ref{double} and the eigenvalues in Lemma \ref{klem}. The result follows.
        
        We have proved (\ref{aij0}). The proofs for (\ref{aij+}) and (\ref{aij-}) are similar, and omitted.
    \end{proof}

    \section{The center of $\mathcal{H}$}
\label{centralH}
In this section we describe the center of $\mathcal{H}$.

The following result is a variation on \cite[Theorem~9.3]{Watanabe}.

\begin{theorem}
\label{central}
    The center of $\mathcal{H}$ is generated by the following three elements:
    \begin{align}
        \Omega_0&=q^{-\frac{h+k}{2}}\Bigl((q-1)F^{0}K_1^{-1}K_2^{-1}+q^{\frac{h}{2}}K_1^{-1}+q^{\frac{k}{2}}K_2^{-1}-K_1^{-1}K_2^{-1}\Bigr),\label{omega0}\\
        \Omega_1&=q^{-\frac{h}{2}}\Biggl(qF^{0}K_2^{-1}+(q-1)F^{-}K_2^{-1}+\frac{q^{\frac{k}{2}+1}K_1K_2^{-1}+q^{\frac{h+k}{2}+1}K_1^{-1}-qK_2^{-1}}{q-1}\Biggr)-\frac{q}{q-1}I,\label{omega1}\\ 
        \Omega_2&=q^{-\frac{k}{2}}\Biggl(qF^0K_1^{-1}+(q-1)F^{+}K_1^{-1}+\frac{q^{\frac{h}{2}+1}K_1^{-1}K_2+q^{\frac{h+k}{2}+1}K_2^{-1}-qK_1^{-1}}{q-1}\Biggr)-\frac{q}{q-1}I.\label{omega2}
    \end{align}
\end{theorem}

\begin{proof}
    We refer to $\Lambda_0,\Lambda_1,\Lambda_2$ from \cite[Section~9]{Watanabe}. We routinely verify that
    \begin{equation*}
        \Omega_0=\Lambda_0^{-1}, \qquad \qquad \Omega_1=q^{\frac{k-1}{2}}(q-1)\Lambda_1-\frac{q+1}{q-1}I,\qquad \qquad \Omega_2=q^{\frac{h-1}{2}}(q-1)\Lambda_2-\frac{q+1}{q-1}I.
    \end{equation*}
    By \cite[Theorem~9.3]{Watanabe}, the elements $\Lambda_0,\Lambda_1,\Lambda_2$ generate the center of $\mathcal{H}$. The result follows.
\end{proof}

\begin{corollary}
\label{omegaev}
    For $0\leq i\leq k$ and $0\leq j\leq h$,
    \begin{equation*}
        \Omega_0E^{*}_{i,j}V\subseteq E^{*}_{i,j}V, \qquad \qquad \Omega_1E^{*}_{i,j}V\subseteq E^{*}_{i,j}V, \qquad \qquad \Omega_2E^{*}_{i,j}V\subseteq E^{*}_{i,j}V.
    \end{equation*}
\end{corollary}

\begin{proof}
    Immediate from (\ref{omega0})--(\ref{omega2}).
\end{proof}

Next we express $F^{0}, F^{+}, F^{-}$ in terms of $\Omega_0,\Omega_1, \Omega_2, K_1^{\pm 1}, K_2^{\pm 1}$.
\begin{lemma}
\label{fcentral}
    We have
    \begin{align}
        F^{0}&=(q-1)^{-1}\Bigl(q^{\frac{h+k}{2}}\Omega_0K_1K_2-q^{\frac{k}{2}}K_1-q^{\frac{h}{2}}K_2+I\Bigr),\label{f0center}\\
        F^{+}&=(q-1)^{-1}\Biggl(q^{\frac{k}{2}}\Omega_2-(q-1)^{-1}\biggl(q^{\frac{h+k}{2}+1}\Bigl(\Omega_0K_2+K_2^{-1}\Bigr)-2q^{\frac{k}{2}+1}I\biggr)\Biggr)K_1,\\
        F^{-}&=(q-1)^{-1}\Biggl(q^{\frac{h}{2}}\Omega_1-(q-1)^{-1}\biggl( q^{\frac{h+k}{2}+1}\Bigl(\Omega_0K_1+K_1^{-1}\Bigr)-2q^{\frac{h}{2}+1}I\biggr)\Biggr)K_2.\label{f-center}
    \end{align}
\end{lemma}

\begin{proof}
    Combine (\ref{omega0})--(\ref{omega2}).
\end{proof}
    
We now find the action of $\Omega_0,\Omega_1,\Omega_2$ on the irreducible $\mathcal{H}$-modules.
\begin{theorem}
\label{omegascalar}
    Let $W$ denote an irreducible $\mathcal{H}$-module of type $(\alpha,\beta,\rho)$. Then each of $\Omega_0$, $\Omega_1$, $\Omega_2$ acts on $W$ as a scalar multiple of the identity. The scalars are given in the table below.
    \begin{center}
        \begin{tabular}{c | c}
            Central element in $\mathcal{H}$ &  Scalar corresponding to $W$\\
            \hline \\
            $\Omega_0$ & $q^{-\rho}$\\ \\
            $\Omega_1$ & $q[k-\rho-\alpha]+[\alpha]$\\ \\
            $\Omega_2$ & $q[h-\rho-\beta]+[\beta]$
        \end{tabular}
    \end{center}
\end{theorem}
    
    \begin{proof}
        We prove the $\Omega_0$-entry of the table. For $\alpha\leq i\leq k-\rho-\alpha$ and $\rho+\beta\leq j\leq h-\beta$, pick a nonzero vector $w\in E^{*}_{i,j}W$. It suffices to show that
        \begin{equation*}
            \Omega_0w=q^{-\rho}w.
        \end{equation*}
        In view of (\ref{omega0}), combine (\ref{aij0}) and Lemma \ref{klem}. The result follows.

        We have proved the $\Omega_0$-entry of the table. The proofs for the other entries are similar, and omitted.
    \end{proof}

We finish this section with a comment.

\begin{lemma}
        The matrices $F^{0}, F^{+}, F^{-}$ mutually commute.
    \end{lemma}

    \begin{proof}
        All the terms on the right-hand sides of (\ref{f0center})--(\ref{f-center}) mutually commute. The result follows.
    \end{proof}

    \section{A subalgebra $\overline{\mathcal{H}}$}

    In this section we define a subalgebra $\overline{\mathcal{H}}$ of $\mathcal{H}$. We find some matrices that are contained in $\overline{\mathcal{H}}$. 
    We describe the action of these matrices on the irreducible $\mathcal{H}$-modules.

    \begin{definition}
        Define 
        \begin{equation*}
            \overline{\mathcal{H}}=\sum_{\ell=0}^{h+k}E^{*}_{\ell}\mathcal{H}E^{*}_{\ell}.
        \end{equation*}
        Note that $\overline{\mathcal{H}}$ is a subalgebra of $\mathcal{H}$.
    \end{definition}

    Our next goal is to find some matrices that are contained in $\overline{\mathcal{H}}$. Note that any diagonal matrix in $\mathcal{H}$ is contained in $\overline{\mathcal{H}}$.

    \begin{lemma}
        The matrices $K_1^{\pm1}, K_2^{\pm1}$ are contained in $\overline{\mathcal{H}}$.
    \end{lemma}
    \begin{proof}
        The result follows from the fact that $K_1^{\pm1}, K_2^{\pm1}$ are diagonal.
    \end{proof}

    \begin{definition}
        Define
        \begin{equation}
        \label{rldef}
            R=L_1R_2,\qquad \qquad \qquad L=L_2R_1.
        \end{equation}
        By Lemma \ref{appendix2} in the appendix, 
        \begin{equation}
        \label{rldef2}
            R=R_2L_1,\qquad \qquad \qquad L=R_1L_2.
        \end{equation}
        Note that $R,L\in \mathcal{H}$. Also note that $R=L^{t}$.
    \end{definition}

     \begin{lemma}
     \label{rllem}
        For $u,v\in P$, the $(u,v)$-entries of $R,L$ are
        \begin{align*}
            R_{u,v}&=\begin{cases}
                1&\text{if $u$ $\backslash$-covers $u\cap v$ and $v$ $\slash$-covers $u\cap v$,}\\
                0&\text{otherwise,}
            \end{cases}\\
            L_{u,v}&=\begin{cases}
                1&\text{if $u$ $\slash$-covers $u\cap v$ and $v$ $\backslash$-covers $u\cap v$,}\\
                0&\text{otherwise.}
            \end{cases}
        \end{align*}
    \end{lemma}

    \begin{proof}
        Immediate from (\ref{rldef}).
    \end{proof}

    \begin{lemma}
    \label{rle}
        For $0\leq i\leq k$ and $0\leq j\leq h$,
        \begin{equation}
        \label{rlev}
            RE^{*}_{i,j}V\subseteq E_{i-1,j+1}^{*}V, \qquad \qquad LE^{*}_{i,j}V\subseteq E_{i+1,j-1}^{*}V.
        \end{equation}
    \end{lemma}

    \begin{proof}
        Combine (\ref{lev}), (\ref{rev}), (\ref{rldef}).
    \end{proof}

    \begin{lemma}
        The matrices $R,L$ are contained in $\overline{\mathcal{H}}$.
    \end{lemma}

    \begin{proof}
        Recall that $R,L\in \mathcal{H}$. By (\ref{rlev}), 
        \begin{equation*}
            R=\sum_{\ell=0}^{h+k}E^{*}_{\ell}RE^{*}_{\ell},\qquad \qquad L=\sum_{\ell=0}^{h+k}E^{*}_{\ell}LE^{*}_{\ell}.
        \end{equation*}
        The result follows.
    \end{proof}

    Next we bring in $F^{0}, F^{+}, F^{-}$, $F$. 

    \begin{lemma}
    \label{fcor}
        The matrices $F^{0}, F^{+}, F^{-}, F$ are contained in $\overline{\mathcal{H}}$.
    \end{lemma}
    \begin{proof}
        Recall from Lemma \ref{finh} that $F^{0},F^{+},F^{-}, F\in \mathcal{H}$. By Lemma \ref{fe*},  \begin{align}
            &F^{0}=\sum_{\ell=0}^{h+k}E^{*}_{\ell}F^{0}E^{*}_{\ell}, &F^{+}=\sum_{\ell=0}^{h+k}E^{*}_{\ell}F^{+}E^{*}_{\ell}, \label{fev1}\\
            &F^{-}=\sum_{\ell=0}^{h+k}E^{*}_{\ell}F^{-}E^{*}_{\ell}, &F=\sum_{\ell=0}^{h+k}E^{*}_{\ell}FE^{*}_{\ell}.\label{fev}
        \end{align}
        The result follows.
    \end{proof}
    
    Next we define a matrix $A$.

    \begin{definition}
    \label{atildelem}
        Define $A\in \text{Mat}_{P}(\mathbb{C})$ as follows. For $u,v\in P$ the $(u,v)$-entry of $A$ is
        \begin{equation*}
            A_{u,v}=\begin{cases}
                1&\text{if each of $u,v$ covers $u\cap v$,}\\
                0&\text{otherwise.}
            \end{cases}
        \end{equation*}
    \end{definition}

    \begin{lemma}
        We have
        \begin{equation}
            \label{adef}
            A=R+L+F.
        \end{equation}
    \end{lemma}

    \begin{proof}
        Combine (\ref{fdef}), (\ref{rldef})  and Definition \ref{atildelem}. The result follows from linear algebra.
    \end{proof}

    \begin{lemma}
        We have
        \begin{align}
            A&=\bigl(L_1+L_2\bigr)\bigl(R_1+R_2\bigr)-(q-1)^{-1}\Bigl(q^{\frac{h+k}{2}}K_1K_2^{-1}-I\Bigr),\label{adef2}\\
            &=\bigl(R_1+R_2\bigr)\bigl(L_1+L_2\bigr)-(q-1)^{-1}\Bigl(q^{\frac{h+k}{2}}K_1^{-1}K_2-I\Bigr).\label{adef3}
        \end{align}
    \end{lemma}

    \begin{proof}
         Equation (\ref{adef2}) follows from (\ref{fdef2}), (\ref{rldef}), (\ref{adef}). Equation (\ref{adef3}) follows from (\ref{fdef3}), (\ref{rldef}), (\ref{adef}).
    \end{proof}

    \begin{lemma}
    \label{acor}
        The matrix $A$ is contained in $\overline{\mathcal{H}}$.
    \end{lemma}

    \begin{proof}
        Immediate from (\ref{adef}).
    \end{proof}

    \begin{lemma}
        For $0\leq i\leq k$ and $0\leq j\leq h$,
        \begin{equation}
        \label{av}
            AE^{*}_{i,j}V\subseteq E_{i+1,j-1}^{*}V+E^{*}_{i,j}V+E_{i-1,j+1}^{*}V.
        \end{equation}
    \end{lemma}

    \begin{proof}
        In view of (\ref{adef}), combine (\ref{rlev}) and the right equation in (\ref{fev}).
    \end{proof}
    
    Next we define a matrix $A^*$.

    \begin{definition}
    \label{a*def}
        Define
        \begin{equation*}
            A^*=q^{\frac{k}{2}}K_1^{-1}.
        \end{equation*}
    \end{definition}

    \begin{lemma}
    \label{a*entry}
        For $u\in P$ the $(u,u)$-entry of $A^*$ is
        \begin{equation*}
            \bigl(A^*\bigr)_{u,u}=q^{i},
        \end{equation*}   
        where $u\in P_{i,j}$.
        \end{lemma}
        \begin{proof}
            Immediate from Definitions \ref{kdef}, \ref{a*def}.
        \end{proof}

     \begin{lemma}
        The matrix $A^*$ is contained in $\overline{\mathcal{H}}$.
    \end{lemma}

    \begin{proof}
        By Definition \ref{a*def}, the matrix $A^*$ is diagonal, and contained in $\mathcal{H}$. The result follows.
    \end{proof}

    Next we bring in $\Omega_0, \Omega_1, \Omega_2$.

    \begin{lemma}
        The matrices $\Omega_0,\Omega_1, \Omega_2$ are contained in $\overline{\mathcal{H}}$.
    \end{lemma}
    
    \begin{proof}
        For each of (\ref{omega0})--(\ref{omega2}), we observe that every term on the right hand side is contained in $\overline{\mathcal{H}}$. The result follows.
    \end{proof}

    \section{The action of $R, L, A, A^*$ on the irreducible $\mathcal{H}$-modules}

    In this section we describe the action of $R, L, A, A^*$ on the irreducible $\mathcal{H}$-modules. For the rest of this section, let $W$ denote an irreducible $\mathcal{H}$-module of type $(\alpha,\beta,\rho)$. 
    \begin{lemma}
    \label{rlewlem}
        For $\alpha\leq i\leq k-\rho-\alpha$ and $\rho+\beta\leq j\leq h-\beta$,
        \begin{equation}
        \label{rlew}
            RE^{*}_{i,j}W\subseteq E_{i-1,j+1}^{*}W, \qquad \qquad LE^{*}_{i,j}W\subseteq E_{i+1,j-1}^{*}W.
        \end{equation}
    \end{lemma}

    \begin{proof}
        Immediate from (\ref{rlev}).
    \end{proof}
    
    Recall from (\ref{wnm}) a standard basis for $W$:
    \begin{equation*}
        w_{i,j}\qquad \qquad \alpha\leq i\leq k-\rho-\alpha\qquad \qquad \rho+\beta\leq j\leq h-\beta.
    \end{equation*}

    \begin{lemma}
        For $\alpha\leq i\leq k-\rho-\alpha$ and $\rho+\beta\leq j\leq h-\beta$,
        \begin{equation}
        \label{rw}
            Rw_{i+1,j-1}=c_{i,j}(W)w_{i,j},\qquad \qquad Lw_{i-1,j+1}=b_{i,j}(W)w_{i,j}, 
        \end{equation}
        where 
        \begin{align}
                c_{i,j}(W)&=q^{\rho+\alpha+\beta}[j-\rho-\beta][k-\rho-\alpha-i],\label{cliw}\\
                b_{i,j}(W)&=q^{k-\rho-i+j+1}[i-\alpha][h-\beta-j].\label{bliw}
            \end{align}
    \end{lemma}
    \begin{proof}
        By (\ref{rlew}), the equations in (\ref{rw}) hold for some $c_{i,j}(W), b_{i,j}(W)\in \mathbb{C}$. We now prove (\ref{cliw}). In view of the left equation of (\ref{rldef}), combine (\ref{l1w}) and (\ref{r2w}). The result follows.
        We have proved (\ref{cliw}). The proof of (\ref{bliw}) is similar, and omitted.
    \end{proof}

    Next we describe the action of $A, A^*$ on $W$.
    
    \begin{lemma}
        For $\alpha\leq i\leq k-\rho-\alpha$ and $\rho+\beta\leq j\leq h-\beta$,
        \begin{equation*}
            AE^{*}_{i,j}W\subseteq E_{i+1,j-1}^{*}W+E^{*}_{i,j}W+E_{i-1,j+1}^{*}W.
        \end{equation*}
    \end{lemma}

    \begin{proof}
        Immediate from (\ref{av}).
    \end{proof}

    \begin{lemma}
        For $\alpha\leq i\leq k-\rho-\alpha$ and $\rho+\beta\leq j\leq h-\beta$,
        \begin{equation*}
            Aw_{i,j}=b_{i+1,j-1}(W)w_{i+1,j-1}+a_{i,j}(W)w_{i,j}+c_{i-1,j+1}(W)w_{i-1,j+1}.
        \end{equation*}
        In the above equation, we interpret
    \begin{equation*}
        w_{r,s}=0,\qquad \qquad b_{r,s}(W)=0,\qquad \qquad c_{r,s}(W)=0,
    \end{equation*}
    unless $\alpha\leq r\leq k-\rho-\alpha$ and $\rho+\beta\leq s\leq h-\beta$.
    \end{lemma}
    \begin{proof}
        In view of (\ref{adef}), combine (\ref{rw}) and the definition of $a_{i,j}(W)$.
    \end{proof}

    \begin{lemma}
        For $\alpha\leq i\leq k-\rho-\alpha$ and $\rho+\beta\leq j\leq h-\beta$, the subspace $E^*_{i,j}W$ is invariant under $A^*$. The corresponding eigenvalue is $q^{i}$.
    \end{lemma}

    \begin{proof}
        Immediate from Lemma \ref{a*entry}.
    \end{proof}

    \section{A generalization of the Askey-Wilson relations}
    Recall the matrix $A$ from Definition \ref{atildelem} and the matrix $A^{*}$ from Definition \ref{a*def}. In this section, we show that $A, A^{*}$ satisfy a pair of relations that generalize the Askey-Wilson relations \cite[Theorem~1.5]{Vidunas}.

    \begin{theorem}
    \label{askeywilson}
        The matrices $A,A^*$ satisfy
        \begin{align}
            &A^2A^{*}-\bigl(q+q^{-1}\bigr)AA^{*}A+A^{*}A^{2}-\mathcal{Y}\bigl(AA^{*}+A^{*}A\bigr)-\mathcal{P} A^{*}=\Omega A+G,\label{askey1}\\
            &A^{*2}A-\bigl(q+q^{-1}\bigr) A^*AA^*+AA^{*2}=\mathcal{Y}A^{*2}+\Omega A^{*}+G^{*},\label{askey2}
        \end{align}
        where 
        \begin{align*}
            \mathcal{Y}=&\;q^{\frac{h+k}{2}}\bigl(K_1K_2^{-1}+K_1^{-1}K_2\bigr)-q^{-1}(q-1)I,\\
            \mathcal{P}=&\;q(q-1)^{-2}\Bigl(\mathcal{Y}^2-q^{h+k-2}(q+1)^2I\Bigr),\\
            \Omega=&\;-q^{\frac{h+k}{2}-1}K_1^{-1}K_2\Bigl((q-1)\Omega_1+(q+1)I\Bigr)-q^{k-1}\Bigl((q-1)\Omega_2+(q+1)I\Bigr),\\
            G=&\;-(q-1)^{-1}\biggl(q^{\frac{h+k}{2}-1}\Bigl(qK_1^{-1}K_2\mathcal{Y}-q^{\frac{h+k}{2}}(q+1)I\Bigr)\Omega_1+q^{k-1}\Bigl(q\mathcal{Y}-q^{\frac{h+k}{2}}(q+1)K_1^{-1}K_2\Bigr)\Omega_2\biggr)\\
            &\qquad \qquad -(q+1)(q-1)^{-2}\biggl(\Bigl(q^{\frac{h+k}{2}}K_1^{-1}K_2+q^k I\Bigr)\mathcal{Y}-q^{\frac{h+k}{2}-1}(q+1)\Bigl(q^{k}K_1^{-1}K_2+q^{\frac{h+k}{2}}I\Bigr)\biggr),\\
            G^*=&\;q^{\frac{h+3k}{2}-1}(q+1)\Omega_0K_1^{-1}K_2.
        \end{align*}
    \end{theorem}
    \begin{proof}
        To verify (\ref{askey1}), (\ref{askey2}) we eliminate $A^*$ using Definition \ref{a*def}. We use (\ref{fdef}), (\ref{f0center})--(\ref{f-center}), (\ref{adef}) to write $A$ in terms of $R, L, K_1^{\pm1}, K_2^{\pm1}, \Omega_0, \Omega_1, \Omega_2$. We evaluate the result using the relations in the appendix.
    \end{proof}

    \begin{lemma}
        The matrices $\mathcal{Y}, \mathcal{P}, \Omega, G, G^*$ are central in $\overline{\mathcal{H}}$.
    \end{lemma}

    \begin{proof}
        Observe that each of $K_1^{-1}K_2$, $K_1K_2^{-1}$, $\Omega_0, \Omega_1, \Omega_2$ is central in $\overline{\mathcal{H}}$. The result follows.
    \end{proof}

    \begin{note}
        The relations (\ref{askey1}), (\ref{askey2}) are generalizations of the Askey-Wilson relations that appear in \cite[Theorem~1.5]{Vidunas}.
    \end{note}

    Next we describe the action of $\mathcal{Y}, \mathcal{P}, \Omega, G, G^*$ on the irreducible $\mathcal{H}$-modules.

    \begin{theorem}
        Let $W$ denote an irreducible $H$-module of type $(\alpha,\beta,\rho)$. Then for $\alpha\leq i\leq k-\rho-\alpha$ and $\rho+\beta\leq j\leq h-\beta$, the subspace $E^{*}_{i,j}W$ is invariant under each of $\mathcal{Y}, \mathcal{P}, \Omega, G, G^*$. The corresponding eigenvalues are given in the table below. 
    \begin{center}
        \begin{tabular}{c | c}
            Element in $\overline{\mathcal{H}}$ &  Eigenvalue corresponding to $E^{*}_{i,j}W$\\
            \hline \\
            $\mathcal{Y}$ & $q^{h+k-\ell}+q^{\ell}-1+q^{-1}$\\ \\
            $\mathcal{P}$ & $q(q-1)^{-2}\Bigl(\bigl(q^{h+k-\ell}+q^{\ell}-1+q^{-1}\bigr)^2-q^{h+k-2}(q+1)^2\Bigr)$\\ \\
            $\Omega$ & $-\bigl(q^{h+k-\rho-\beta}+q^{k+\beta-1}+q^{k+\ell-\rho-\alpha}+q^{\ell+\alpha-1}\bigr)$\\ \\
            $G$ & $\substack{(q-1)^{-1}\Bigl(\bigl(q^{k-\rho-\alpha+1}+q^{\alpha}\bigr)\bigl(q^{\ell-1}[h+k-\ell]-q^{\ell}[\ell]\bigr)-\bigl(q^{h-\rho-\beta+1}+q^{\beta}\bigr)\bigl(q^k[h+k-\ell]-q^{k-1}[\ell]\bigr)\Bigr)}$\\ \\
            $G^*$ & $q^{k+\ell-\rho-1}(q+1)$
        \end{tabular}
    \end{center}
    In the above table, we write $\ell=i+j$.
    \end{theorem}

    \begin{proof}
        Combine Lemma \ref{klem} and Theorem \ref{omegascalar} and the definitions of $\mathcal{Y}, \mathcal{P},\Omega,G,G^*$.
    \end{proof}

    \section{Parameters $\nu,\mu,d,e$}
    Let $W$ denote an irreducible $\mathcal{H}$-module of type $(\alpha,\beta,\rho)$. In some earlier papers in the literature, the action of $\overline{\mathcal{H}}$ on $E^{*}_{k}W$ is often expressed in terms of $\nu,\mu,d,e$ where
    \begin{align*}
        \nu&=\min\{j\mid E^{*}_{k-j,j}W\neq 0,\; 0\leq j\leq k\},\\
        \mu&=\min\{i+j\mid E^{*}_{i,j}W\neq 0,\; 0\leq i\leq k,\; 0\leq j\leq k\},\\
        d&=\bigl \vert \{j\mid E^{*}_{k-j,j}W\neq 0,\; 0\leq j\leq k\}\bigr \vert-1,
        \end{align*}
    and $e$ is an auxiliary parameter. (See \cite{Liang}, \cite[Example~6.1]{ter1}.) The parameters $\nu,\mu,d$ are often called the endpoint, dual endpoint, and diameter, respectively. In this section, we describe how to convert from $\alpha,\beta,\rho$ to $\nu,\mu,d,e$.
    
    There are three cases to consider:
    \begin{enumerate}[label=(C\arabic*)]
        \item $\beta-\alpha\leq 0$;
        \item $0<\beta-\alpha\leq h-k$;
        \item $h-k< \beta-\alpha$.
    \end{enumerate}
    We illustrate the cases (C1)--(C3) using the diagrams below.

    \begin{figure}[!ht]
    \centering
{%
\begin{circuitikz}
\tikzstyle{every node}=[font=\normalsize]
\draw [short] (2.5,-3.5) -- (-1,0);
\draw [short] (-1,0) -- (1.5,2.5);
\draw [short] (1.5,2.5) -- (5,-1);
\draw [short] (5,-1) -- (2.5,-3.5);
\draw [ color=red , fill=pink, rotate around={45:(2,2)}] (0.5,2.25) rectangle (-0.25,-1.5);
\draw [dashed] (5,-1) -- (0, -1);
\draw [<->, >=Stealth] (2.9,-0.9) -- (4.85,-0.9);
\draw [<->, >=Stealth] (1.95,-1.1) -- (2.9,-1.1);
\draw [short] (2.8,-2.09) -- (3,-2.09);
\draw [short] (3,-3.5) -- (2.4,-3.5);
\draw [<->, >=Stealth] (2.9,-2.09) -- (2.9,-3.5);
\node [font=\scriptsize] at (3.95,-0.75) {$\nu$};
\node [font=\scriptsize] at (2.7,-2.8) {$\mu$};
\node [font=\scriptsize] at (2.45,-1.25) {$d$};
\node [font=\normalsize] at (1.5,2.85) {$\mathcal{V}$};
\node [font=\normalsize] at (5.25,-1) {$y$};
\node [font=\normalsize] at (2.5,-3.75) {$0$};
\node [font=\normalsize] at (1.75,0.75) {$W$};
\node [font=\normalsize] at (2,-4.5) {Figure 4.1: Case (C1)};
\end{circuitikz}
\hspace{0.25cm}
\begin{circuitikz}
\tikzstyle{every node}=[font=\normalsize]
\draw [short] (2.5,-3.5) -- (-1,0);
\draw [short] (-1,0) -- (1.5,2.5);
\draw [short] (1.5,2.5) -- (5,-1);
\draw [short] (5,-1) -- (2.5,-3.5);
\draw [ color=red , fill=pink, rotate around={45:(2,2)}] (0.5,1.25) rectangle (-0.25,-0.5);
\draw [dashed] (5,-1) -- (0, -1);
\draw [<->, >=Stealth] (2.55,-1.1) -- (4.85,-1.1);
\draw [<->, >=Stealth] (1.75,-0.9) -- (2.55,-0.9);
\draw [short] (2,-1.35) -- (2.3,-1.35);
\draw [short] (2,-3.5) -- (2.6,-3.5);
\draw [<->, >=Stealth] (2.15,-1.35) -- (2.15,-3.5);
\node [font=\scriptsize] at (3.75,-1.3) {$\nu$};
\node [font=\scriptsize] at (2.35,-2.4) {$\mu$};
\node [font=\scriptsize] at (2.15,-0.75) {$d$};
\node [font=\normalsize] at (1.5,2.85) {$\mathcal{V}$};
\node [font=\normalsize] at (5.25,-1) {$y$};
\node [font=\normalsize] at (2.5,-3.75) {$0$};
\node [font=\normalsize] at (1.75,0.75) {$W$};
\node [font=\normalsize] at (2,-4.5) {Figure 4.2: Case (C2)};
\end{circuitikz}
}%
\end{figure}

\begin{figure}[!ht]
    \centering
    {%
    \begin{circuitikz}
\tikzstyle{every node}=[font=\normalsize]
\draw [short] (2.5,-3.5) -- (-1,0);
\draw [short] (-1,0) -- (1.5,2.5);
\draw [short] (1.5,2.5) -- (5,-1);
\draw [short] (5,-1) -- (2.5,-3.5);
\draw [ color=red , fill=pink, rotate around={45:(2,2)}] (1.4,0.65) rectangle (-1.15,-0.3);
\draw [dashed] (5,-1) -- (0, -1);
\draw [<->, >=Stealth] (2.3,-1.1) -- (4.85,-1.1);
\draw [<->, >=Stealth] (1.05,-0.9) -- (2.3,-0.9);
\draw [short] (1.3,-1.85) -- (1.5,-1.85);
\draw [short] (1.3,-3.5) -- (2.6,-3.5);
\draw [<->, >=Stealth] (1.4,-1.85) -- (1.4,-3.5);
\node [font=\scriptsize] at (3.55,-1.3) {$\nu$};
\node [font=\scriptsize] at (1.2,-2.7) {$\mu$};
\node [font=\scriptsize] at (1.7,-0.75) {$d$};
\node [font=\normalsize] at (1.5,2.85) {$\mathcal{V}$};
\node [font=\normalsize] at (5.25,-1) {$y$};
\node [font=\normalsize] at (2.5,-3.75) {$0$};
\node [font=\normalsize] at (1.9,0.6) {$W$};
\node [font=\normalsize] at (2,-4.5) {Figure 4.3: Case (C3)};
\end{circuitikz}
}%
\end{figure}

\renewcommand{\thefigure}{\arabic{figure}}

\setcounter{figure}{6}
\newpage 

\begin{lemma}{\cite[p.~129]{Liang}}
\label{inverse}
    Below we express $\alpha, \beta, \rho$ in terms of $\nu, \mu, d,e$. 
    
    In case (C1),
    \begin{equation*}
        \alpha=\nu-e,\qquad \qquad \beta=\mu-\nu,\qquad \qquad \rho=e.
    \end{equation*}
    In case (C2),
    \begin{equation*}
        \alpha=\mu-\nu,\qquad \qquad \beta=\frac{\mu-e}{2},\qquad \qquad \rho=\nu-\frac{\mu-e}{2}.
    \end{equation*}
    In case (C3),
    \begin{equation*}
        \alpha=\mu-\nu,\qquad \qquad \beta=k-h+\nu-e,\qquad \qquad \rho=h-k+e.
    \end{equation*}
\end{lemma}

\section{Appendix}
In this section, we recall from \cite{Watanabe} the relations between the generators of $\mathcal{H}$.

\begin{lemma}{\cite[Lemma~7.4]{Watanabe}}
\label{appendix1}
    The following (i)--(viii) hold:
    \begin{enumerate}[label=(\roman*)]
        \item $K_1L_1=qL_1K_1$;
        \item $K_1L_2=L_2K_1$;
        \item $qK_1R_1=R_1K_1$;
        \item $K_1R_2=R_2K_1$;
        \item $K_2L_1=L_1K_2$;
        \item $qK_2L_2=L_2K_2$;
        \item $K_2R_1=R_1K_2$;
        \item $K_2R_2=qR_2K_2$.
    \end{enumerate}
\end{lemma}

\begin{lemma}{\cite[Lemma~7.5]{Watanabe}}
\label{appendix2}
    The following (i)--(iv) hold:
    \begin{enumerate}[label=(\roman*)]
        \item $L_1R_2=R_2L_1$;
        \item $L_2R_1=R_1L_2$;
        \item $qL_1L_2=L_2L_1$;
        \item $R_1R_2=qR_2R_1$.
    \end{enumerate}
\end{lemma}

\begin{lemma}{\cite[Lemma~7.6]{Watanabe}}
\label{appendix3}
    The following (i)--(iv) hold.
    \begin{enumerate}[label=(\roman*)]
        \item $R_1^2L_1-(q+1)R_1L_1R_1+qL_1R_1^2=-q^{\frac{h+k}{2}-1}(q+1)K_1^{-1}K_2R_1$;
        \item $qR_2^2L_2-(q+1)R_2L_2R_2+L_2R_2^2=-q^{\frac{h+k}{2}}(q+1)K_1K_2^{-1}R_2$;
        \item $qL_1^2R_1-(q+1)L_1R_1L_1+R_1L_1^2=-q^{\frac{h+k}{2}}(q+1)K_1^{-1}K_2L_1$;
        \item $L_2^2R_2-(q+1)L_2R_2L_2+qR_2L_2^2=-q^{\frac{h+k}{2}-1}(q+1)K_1K_2^{-1}L_2$.
    \end{enumerate}
\end{lemma}

\begin{lemma}{\cite[Lemma~7.7]{Watanabe}}
\label{appendix4}
    The generators of $\mathcal{H}$ satisfy
    \begin{equation*}
        L_1R_1-R_1L_1+L_2R_2-R_2L_2=q^{\frac{h+k}{2}}(q-1)^{-1}(K_1K_2^{-1}-K_1^{-1}K_2).
    \end{equation*}
\end{lemma}

\section*{Acknowledgement}
The author is currently a graduate student at the University of Wisconsin-Madison. He would like to thank his advisor, Paul Terwilliger, for all the valuable ideas and suggestions during the preparation of this manuscript.

Ian Seong\\
Department of Mathematics\\
University of Wisconsin \\
480 Lincoln Drive \\
Madison, WI 53706-1388 USA \\
email: iseong@wisc.edu\\
\end{document}